\documentclass[
final
 , nomarks
]{dmtcs-episciences}

\usepackage[utf8]{inputenc}
\usepackage{subfigure}

\usepackage[round]{natbib}
\usepackage{amsmath,amssymb,amsthm,tikz}
\usepackage{tkz-berge}
\usetikzlibrary{graphs, patterns}
\usepackage{enumerate}
\usepackage[ruled,vlined]{algorithm2e}
\usepackage{lineno}

\newtheorem{theorem}{Theorem}
\newtheorem{definition}[theorem]{Definition}

\newtheorem{conjecture}[theorem]{Conjecture}

\newtheorem{observation}[theorem]{Observation}
\newtheorem{proposition}[theorem]{Proposition}
\newtheorem{question}[theorem]{Question}

\newtheorem*{experiment}{Experimental results}

\newcommand{\Av}{\operatorname{Av}}

\newcommand{\Grid}{\operatorname{Grid}}
\newcommand{\Inv}{\operatorname{Inv}}

\newcommand{\we}{\equiv}

\newcommand{\C}{{\mathcal{C}}}
\newcommand{\D}{{\mathcal{D}}}
\newcommand{\E}{{\mathcal{E}}}
\newcommand{\F}{{\mathcal{F}}}
\newcommand{\I}{{\mathcal{I}}}
\renewcommand{\S}{{\mathcal{S}}}

\author{Michael Albert\affiliationmark{1} \and Jinge Li\affiliationmark{1}}
\title[Uniquely-Wilf classes]{Uniquely-Wilf classes}
\affiliation{Department of Computer Science, University of Otago}
\received{2019-4-12}
\revised{2019-9-25}
\accepted{2019-10-15}

\begin{document}
\publicationdetails{21}{2019}{2}{7}{5374}
\maketitle
\begin{abstract}
Two permutations in a class are Wilf-equivalent if, for every size, $n$, the number of permutations in the class of size $n$ containing each of them is the same. Those infinite classes that have only one equivalence class in each size for this relation are characterised provided either that they avoid at least one permutation of size 3, or at least three permutations of size 4.
\end{abstract}

\section{Introduction}

In the classical study of permutation patterns, two permutations $\pi$ and $\sigma$ are said to be \emph{Wilf-equivalent} if the class, $\Av(\pi)$, of permutations avoiding $\pi$ and the class, $\Av(\sigma)$, of permutations avoiding $\sigma$ have the same enumeration sequence. Notational definitions are given in the next section. Arguably, one of the main results that initiated this area of study was that any two permutations of size 3 are Wilf-equivalent. Due to underlying symmetries this relation could have had at most two equivalence classes on permutations of size 3. It has since been discovered that among the permutations of size 4 there are only three Wilf-equivalence classes while, taking only symmetry into account, there could have been up to seven.

In this paper we are concerned with a relativised version of Wilf-equivalence. Rather than considering all the permutations that avoid $\pi$ we consider only those that belong to some given class $\C$ -- so the original notion just corresponds to taking $\C$ to be the class of all permutations. 

\begin{definition}
Let $\C$ be a permutation class. Two permutations $\pi, \sigma \in \C$ are \emph{Wilf-equivalent relative to $\C$} if the two classes $\C \cap \Av(\pi)$ and $\C \cap \Av(\sigma)$ have the same enumeration sequence. In this case we write $\pi \we_{\C} \sigma$. 
\end{definition}

If $\C$ is clear from context we sometimes omit mentioning it.

Let $(c_n)_{n \in \naturals}$ be the enumeration sequence of a class $\C$. If $\pi \in \C$ has size $k$ then the enumeration sequence $(p_n)_{n \in \naturals}$ of $\C \cap \Av(\pi)$ satisfies $p_n = c_n$ for $n < k$ and $p_k = c_k - 1$. Therefore, if $\pi$ and $\sigma$ are Wilf-equivalent relative to $\C$ they must have the same size. 

\begin{definition}
Let $\C$ be a permutation class. The \emph{Wilf-sequence} of $\C$ is the sequence $(w_n)_{n \in \naturals}$ where $w_n$ is the number of $\we_\C$-equivalence classes among the permutations of $\C$ of size $n$.
\end{definition}

Obviously $w_n \leqslant c_n$ in general but it has been observed in multiple contexts that it is frequently the case that $w_n = o(c_n)$. We refer to this phenomenon as a \emph{Wilf-collapse}. Since it is known that the partially ordered sets corresponding to permutation classes do not have infinite automorphism groups (\cite{AAC15}) this cannot occur solely due to global symmetries of the containment relation within a pattern class. Our particular concern in this paper is with the ultimate Wilf-collapse:

\begin{definition}
A permutation class $\C$ is \emph{uniquely-Wilf} if the equivalence relation $\we_{\C}$ has only one equivalence class on permutations of size $n$ in $\C$ for all $n$.
\end{definition}

In \cite{AlbertLi} it was shown that the class $\Av(312, 213)$ is uniquely-Wilf. It is the aim of this paper to explore the landscape of uniquely-Wilf classes with the ultimate (but unrealised) goal of classifying them completely.

\section{Preliminaries}

In this section we introduce our notation, definitions and make some basic observations. Notations chosen are somewhat standard in the field of classical permutation patterns so we shall be quite brief. A primer on this notation can be found in \cite{BevanNotes} and for introductions to and/or surveys of the field we suggest the books \cite{BonaBook} and \cite{KitaevBook} and a survey article \cite{VatterSurvey}.

We denote by $\S$ the class of all finite permutations. For a set of permutations $X$ we denote the set of permutations in $X$ of size $k$ by $X_k$ and similarly use $X_{\leq k}$ etc. The \emph{enumeration sequence} $(x_n)_{n \in \mathbb{N}}$ of $X$ is defined by $x_k = | X_k |$. 

The set $\S$ carries a partial order $\leq$ called \emph{classical pattern containment}. Specifically, $\pi \leq \sigma$ if $\sigma$, considered as a sequence in one line notation, contains a subsequence order-isomorphic to $\pi$. There are eight automorphisms of $\S$ as a partially ordered set generated by the operations on sequences of: reverse, complement (subtract each term for a permutation of size $n$ from $n+1$) and inverse. All the questions we consider respect these automorphisms.

If $\pi \leq \sigma$ we say that $\sigma$ \emph{involves} $\pi$, while if $\pi \not \leq \sigma$ we say that $\sigma$ \emph{avoids} $\pi$. Further we define:
\begin{align*}
\Inv(\pi) &= \{ \sigma \in \S \, : \, \pi \leq \sigma \} \\
\Av(\pi) &= \{ \sigma \in \S \, : \, \pi \not \leq \sigma \}.
\end{align*}
When $\pi \leq \sigma$ and $|\sigma| = |\pi| + 1$ we say that $\sigma$ \emph{covers} $\pi$.

The $\Av(\pi)$ notation is extended to sets of permutations with the requirement that all the permutations in the set must be avoided. A set of permutations is a \emph{class} if it is downwards-closed with respect to $\leq$ or equivalently if it is of the form $\Av(X)$ for some set of permutations $X$. As already seen we generally use Greek letters to denote individual permutations, upper case Latin letters to denote sets of permutations, and calligraphic upper case Latin letters to denote permutation classes. Additionally we set $\I = \Av(21)$ to be the class of all monotone-increasing permutations and $\D = \Av(12)$ to be the class of all monotone-decreasing permutations.

\begin{definition}
Let $k$ and $n$ be positive integers. A permutation class $\C$ is \emph{$(k,n)$-balanced} if, for all $\pi, \sigma \in \C_k$, $|\C_n \cap \Inv(\pi)| = |\C_n \cap \Inv(\sigma)|$.
\end{definition}

Any permutation class is $(k,n)$-balanced for any $k \geq n$ (since the two sets in question are singletons if $k = n$ and empty otherwise) so the issue of whether or not a class is $(k,n)$-balanced is really of interest only if $k < n$.

\begin{observation}
A permutation class $\C$ is uniquely-Wilf if and only if it is $(k,n)$-balanced for all pairs of positive integers $(k,n)$.
\end{observation}

\begin{proof}
Suppose that $\C$ is uniquely-Wilf and that $\pi, \sigma \in \C_k$. Then $\pi \we_{\C} \tau$ so the two classes $\C \cap \Av(\pi)$ and $\C \cap \Av(\sigma)$ have the same enumeration sequence. In particular $|\C_n \cap \Av(\pi)| = | \C_n \cap \Av(\sigma)|$. But $\C_n \cap \Inv(\pi) = \C_n \setminus (\C_n \cap \Av(\pi))$ and similarly for $\sigma$ so this is equivalent to $|\C_n \cap \Inv(\pi)| = | \C_n \cap \Inv(\sigma) |$, i.e., $\C$ is $(k,n)$-balanced.

Conversely, suppose that $\C$ is $(k,n)$-balanced for all pairs of positive integers $(k,n)$. Let $\pi, \sigma \in \C_k$ be given. Then, for any $n$, 
\[
| \Av(\pi) \cap \C_n | = c_n - | \Inv(\pi) \cap \C_n | = c_n - | \Inv(\sigma) \cap \C_n | = | \Av(\sigma) \cap \C_n |,
\]
i.e., $\pi \we_{\C} \sigma$, hence $\C$ is uniquely-Wilf.
\end{proof}

\begin{observation} 
\label{obs:monotone}
Let $\C$ be a uniquely-Wilf class that contains both $12$ and $21$. Then, for each $k$, either $\C$ contains both $12 \cdots k$ and $k (k-1) \cdots 1$ or it contains neither of those permutations. In the latter case $\C$ is finite.
\end{observation}

\begin{proof}
If, for some $k$, $\C$ contained $12 \cdots k$ but not $k (k-1) \cdots 1$ then $12 \not \we_{\C} 21$ since $\C$ contains a permutation of size $k$ avoiding 21 but none avoiding $12$. This contradiction establishes the first claimed result. The second part follows immediately since, by the Erd\H{o}s-Szekeres theorem (\cite{ES}), a class containing neither $12 \cdots k$ nor $k (k-1) \cdots 1$ cannot contain any permutation of size greater than $(k-1)^2$.
\end{proof}

Observation \ref{obs:monotone} establishes a trichotomy for uniquely-Wilf classes: they might be finite, consist only of monotone permutations of a single type, or contain both monotone permutations of every size. Somewhat disingenuously we consider the first two cases as trivial ones and focus on the last. This is disingenuous since any uniquely-Wilf class $\C$ is the union of all the finite uniquely-Wilf classes $\C_{\leq k}$ for each positive integer $k$. This observation is central to the experimental results described in the next section.

\section{Experimental results}

Let $\F$ be a finite uniquely-Wilf class whose permutations are all of size at most $n$. We are interested in determining those uniquely-Wilf classes $\C$ with the property that $\C \cap \S_{\leq n} = \F$. This condition already puts some restrictions on $\C$ -- specifically that any permutation of size at most $n$ that does not belong to $\F$ is avoided by all the elements of $\C$, i.e., $\C \subseteq \Av(\S_{\leq n} \setminus \F)$. For convenience we denote $\Av(\S_{\leq n} \setminus \F)$ by $\F^{\uparrow}$ -- this is the maximum class (in terms of set containment) whose restriction to $\S_{\leq n}$ is $\F$.

\begin{observation}
Let $\C$ be a permutation class and let $k$ be a positive integer. Then $\C = \left( \C_{\leq k} \right)^{\uparrow}$ if and only if $\C = \Av(X)$ for some $X \subseteq \S_{\leq k}$.
\end{observation}

\begin{proof}
This is essentially a restatement of the definition. If $\C =  \left( \C_{\leq k} \right)^{\uparrow}$ then we can take $X$ to be $\Av(\S_{\leq k} \setminus \C_{\leq k})$. Conversely if $\C = \Av(X)$ for some $X \subseteq \S_{\leq k}$ then $X \subseteq \S_{\leq k} \setminus \C_{\leq k}$ and so $\Av(\S_{\leq k} \setminus \C_{\leq k}) \subseteq \C$. But the reverse containment always holds so $\C = \Av(\S_{\leq k} \setminus \C_{\leq k})$.
\end{proof}

If $\F \subseteq \S_{\leq n}$ is a finite uniquely-Wilf class, then an obvious question to address is:
\begin{quote}
\emph{Which subsets $X \subseteq \S_{n+1}$ have the property that $\F \cup X$ is a uniquely-Wilf class?}
\end{quote}

Call such subsets \emph{potential extensions} of $\F$ to size $n+1$. It's quite easy to determine whether $X$ is a potential extension of $\F$ of size $n+1$. First of all $X$ must not contain any permutations that cannot lie in any extension of $\F$, i.e., $X \subseteq \F^{\uparrow}$. This ensures that $\F \cup X$ is a class. Secondly,  $\F \cup X$ must be $(k, n+1)$-balanced for all $k \leq n$. Since $\F$ was already $(k, m)$ balanced for all $m \leq n$ and $\F \cup X$ contains no permutations of size greater than $n+1$ these two conditions are both necessary and sufficient for $X$ to be a potential extension of $\F$.

This suggests an experimental technique to build uniquely-Wilf classes from the bottom up. Start with a finite permutation class $\F$ that is uniquely-Wilf and contains permutations up to size $n$. Consider binary vectors (i.e., indicator functions) $\mathbf{x} : \left( \F^{\uparrow} \right)_{n+1} \to \{0,1\}$ where, as usual we denote $\mathbf{x}(\pi)$ by $\mathbf{x}_\pi$ and associate $\mathbf{x}$ with the set $X = \mathbf{x}^{-1}(1)$. Then \textit{par abus de langage} we call $\mathbf{x}$ a potential extension of $\F$ to size $n+1$ if the associated set $X$ is. The important thing is that this construction already captures the first condition above since we only consider permutations from $\F^{\uparrow}$. Each of the balancing conditions that any two permutations of the same size, say $\sigma, \tau \in \F_k$, have the same number of extensions in $X$ can be expressed as a constraint:
\[
\sum_{\sigma \leq \pi} \mathbf{x}_{\pi} - \sum_{\tau \leq \pi} \mathbf{x}_{\pi} = 0.
\]

Here, and in the following few equations, $\pi$ ranges over $\left( \F^{\uparrow} \right)$ (i.e., the set of permutations on which the notation  $\mathbf{x}_{\pi}$ is defined).

From a practical standpoint we might wish to change the form in which we express these constraints. For instance we could use instead:
\[
\sum_{\sigma \leq \pi, \tau \not \leq \pi} \mathbf{x}_{\pi} - \sum_{\tau \leq \pi, \sigma \not \leq \pi} \mathbf{x}_{\pi} = 0.
\]
Or, for each $2 \leq k \leq n$ we could assign a target, $t_k$, with 
\[
0 \leq t_k \leq \left| \left(\F^{\uparrow} \right)_{n+1} \right|
\]
and, for $\sigma \in \F_k$, use the constraint
\[
\sum_{\sigma \leq \pi} \mathbf{x}_{\pi}  = t_k.
\]
Or we could simply look for all solutions that give $(n, n+1)$-balance (possibly with a target) and then filter out those that do not also give $(k, n+1)$-balance for $k < n$.

Whatever our chosen system of constraints might be it can then be passed to a constraint solver. We used the system Minion (\cite{MINION}). All solutions can be found and then (if necessary) these can be filtered to provide all the potential extensions. Then we can proceed in the same manner on each of these (trying to further extend to size $n+2$) until either some resolution is achieved or we exhaust our patience or machine resources. What sorts of resolution can we hope for?
\begin{itemize}
\item
We might find that none of the potential extensions of $\F$ contain the monotone permutations of size $n+1$. In this case we know that there is no infinite uniquely-Wilf class whose restrictions to permutations of size at most $n$ is $\F$.
\item
We might find that the only potential extension of $\F$ that contains both monotone permutations of size $n+1$ is $\F^{\uparrow}_{n+1}$. If this continues to be the case when we extend to $n+2$, $n+3$ etc.\ then we might form the opinion that $\F^{\uparrow}$ is uniquely-Wilf and is the only such class whose restriction to permutations of size at most $n$ is $\F$. However, this does not constitute a proof of that opinion -- but it is something we can attempt to prove.
\end{itemize}

Since we are interested only in infinite uniquely-Wilf classes containing both 12 and 21 our starting point is the finite class $\S_{\leq 2} = \{1, 12, 21\}$. The $(1,n)$-balancing condition is always trivially satisfied and $(2,n)$-balance is satisfied whenever we include both monotone permutations of size $n$. So our first sets of possible extensions are the union of $\{123, 321\}$ with each of the sixteen subsets of $\{132, 213, 231, 312\}$. Since all questions about Wilf-equivalence are invariant under the action of the automorphism group of $(\S, \leq)$ we can further restrict our experiments to one representative of each orbit of that group on sets of permutations. Our findings are as follows:

\begin{experiment}
Let $\F \subseteq \S_{\leq 3}$ be a finite class with $\F_2 = \{12, 21\}$. Then:
\begin{enumerate}
\item
If $|\F_3| = 2$ the only infinite uniquely-Wilf class containing $\F$ is $\F^{\uparrow} = \I \cup \D$.
\item
If $|\F_3| = 3$ the only infinite uniquely-Wilf class containing $\F$ is $\F^{\uparrow}$.
\item
If $|\F_3| = 4$ and $\F_3 = \{123, 213, 312, 321\}$ or one of its symmetries then the only infinite uniquely-Wilf class containing $\F$ is $\F^{\uparrow}$ while, for all other choices of $\F_3$ with $|\F_3| = 4$, there are no such classes.
\item
If $|\F_3| = 5$  then there is no infinite uniquely-Wilf class containing $\F$.
\item
If $\F_3 = \S_3$ and $|\F_4| < 22$ then there are no infinite uniquely-Wilf classes containing $\F$.
\end{enumerate}
\end{experiment}

In the next section we prove the first and second results above along with the first half of the third result. Provided that we trust our constraint solver (which we do) and our translation of the constraints into code\footnote{This code is freely available from the authors on request. Owing to its specialised nature we do not feel it appropriate, at this time, to make it part of a more widely available repository.} for the constraint solver (which we have to) the second half of the third result, the fourth result, and the fifth result require no further proof. 

Of note is that we have not completely classified infinite uniquely-Wilf classes since the cases where $\F_3 = \S_3$ and $|\F_4| \in \{22, 23, 24\}$ remain open -- this is where exhaustion set in as there were simply too many possible extensions to deal with using our present ideas. While the case $|\F_4| = 22$ could probably be handled by our current techniques using some additional computational power it does seem that the remaining cases are definitely out of reach using the straight forward constraints presented here. In order to complete the classification would seem to require either determining some new constraints that have been overlooked, some clever ordering of the CSP that facilitates its (non-)solution, or a new structural idea.

However, we are confident enough to present:

\begin{conjecture}
There are no infinite uniquely-Wilf classes, $\C$, with $\C_3 = \S_3$.
\end{conjecture}

\section{Proofs}

In this section we provide the required proofs that justify the resolutions of the experiment above.

\begin{proposition}
The only infinite uniquely-Wilf class, $\C$ with $\C_3 = \{123, 321\}$ is $\I \cup \D$.
\end{proposition}

\begin{proof}
Certainly $\I \cup \D$ is uniquely-Wilf. If $\C_3 = \{123, 321\}$ then $\C_{\leq 3}^{\uparrow} = \I \cup \D$ as any non-monotone permutation of size 4 contains at least one non-monotone permutation of size 3. But any infinite uniquely-Wilf class containing 12 and 21 must contain $\I \cup \D$.
\end{proof}

Up to symmetry, there is only one set of three permutations of size three containing $\{123, 321\}$ so the second result we need is:

\begin{proposition}
The only infinite uniquely-Wilf class $\C$ with $\C_3 = \{123, 132 , 321\}$ is $\C = \Av(213, 231, 312)$.
\end{proposition}

\begin{proof}
Let $\C = \Av(213, 231, 312)$. Since all the permutations avoided in $\C$ are of size three, $\C = \C_{\leq 3}^{\uparrow}$. To verify that $\C$ is uniquely-Wilf is easy. The permutations in $\C$ have the structure shown in the picture below:

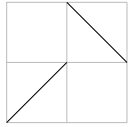
\vspace{\parskip}

\centerline{
\begin{tikzpicture}[scale=0.8]
\draw [color=lightgray, xstep=1.0cm,ystep=1.0cm]
(0,0) grid (2,2);
\draw (0,0) --(1,1);
\draw (1,2) --(2,1);
\end{tikzpicture}
}

That is, they can be represented by an increasing sequence followed by a decreasing one. There is an obvious one to one correspondence between such permutations and pairs $(a,b)$ of integers with $a \geq 0$ and $b \geq 1$. Here $a$ represents the size of the increasing sequence and $b$ the size of the decreasing one -- and we always choose the maximum element to be in the decreasing sequence. Thus, the representation of the increasing permutation of size $n$ in this form is as the pair $(n-1, 1)$. With respect to this representation $(a, b) \leq (c,d)$ as permutations if and only if $a \leq c$ and $b \leq d$ i.e.,  the containment order of permutations and the usual coordinatewise order coincide. In particular each permutation, $(a,b)$, of size $n$ has exactly $k+1$ extensions of size $n+k$ namely those corresponding to pairs $(a+x, b+k-x)$ for $0 \leq x \leq k$. Thus, $\C$ is uniquely-Wilf.

Note that, in $\C$ every permutation has exactly two covers since $(a,b)$ is covered by $(a+1, b)$ and by $(a, b+1)$.

Suppose that $\E$ were a proper uniquely-Wilf subclass of $\C$ containing all monotone permutations and with $\E_3 = \C_3$. Let $n+1$ be the least size of any permutation in $\C \setminus \E$ and note that $n \geq 3$. Since we omit one element of $\C$ from $\E$ in size $n+1$ there is an element of $\E$ of size $n$ with only one cover in $\E$. Therefore, as $\E$ is uniquely-Wilf, every element of $\E$ of size $n$ must have exactly one cover of size $n+1$.  In particular, the monotone permutation $(n, 1)$ must be the unique cover of $(n-1,1)$ belonging to $\E$ and hence its other cover $(n-1,2) \not \in \E$.

As $(n-2, 2) \in \E$, but $(n-1,2) \not \in \E$,  it must be the case that $(n-2, 3) \in \E$ since that is the only remaining cover of $(n-2,2)$ in $\C$. Now if $n = 3$ we already have a contradiction since $(1,3), (0,4) \in \E$ are the two covers of $(0,3) \in \E$. But if $n \geq 4$ then since $(n-3, 4)$ also covers $(n-3,3)$ we have $(n-3,4) \not \in \E$. Since $(n-1,2), (n-3,4) \not \in \E$ none of their covers can lie in $\E$. But these include 
$(n-1, 3)$ and $(n-2,4)$ which are the two covers of $(n-2,2)$. So $(n-2,2)$ has no covers in $\E$ and we have our contradiction (since $\E$ is infinite all monotone permutations have one cover in $\E$). The latter part of the argument is perhaps better understood through the following picture where filled circles represent elements of $\E$ and open ones elements of $\C \setminus \E$:

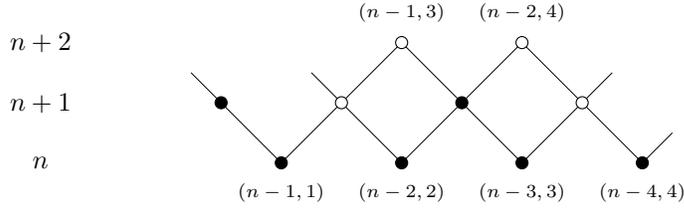
\centerline{
\begin{tikzpicture}[scale=0.8]
\draw[ultra thin] (0,0) -- (-1.5,1.5);
\draw[ultra thin] (0,0) -- (2,2);
\draw[ultra thin] (2,0) -- (0.5,1.5);
\draw[ultra thin] (2,0) -- (4,2);
\draw[ultra thin] (4,0) -- (2,2);
\draw[ultra thin] (4,0) -- (5.5,1.5);
\draw[ultra thin] (6,0) -- (4,2);
\draw[ultra thin] (6,0) -- (6.5,0.5);
\draw [fill] (0,0) circle (0.1);
\draw [fill] (2,0) circle (0.1);
\draw [fill] (4,0) circle (0.1);
\draw [fill] (6,0) circle (0.1);
\draw [fill] (-1,1) circle (0.1);
\draw [fill=white] (1,1) circle (0.1);
\draw [fill] (3,1) circle (0.1);
\draw [fill=white] (5,1) circle (0.1);
\draw [fill=white] (2,2) circle (0.1);
\draw [fill=white] (4,2) circle (0.1);
\node at (-4,0) {$n$};
\node at (-4,1) {$n+1$};
\node at (-4,2) {$n+2$};
\node at (0,-0.5) {\scriptsize$(n-1,1)$};
\node at (2,-0.5) {\scriptsize$(n-2,2)$};
\node at (4,-0.5) {\scriptsize$(n-3,3)$};
\node at (6,-0.5) {\scriptsize$(n-4,4)$};
\node at (2,2.5) {\scriptsize$(n-1,3)$};
\node at (4,2.5) {\scriptsize$(n-2,4)$};
\end{tikzpicture}
}
\end{proof}

The second half of the proof above can also be adapted to characterise the finite uniquely-Wilf classes containing 123, 321 and one other permutation of size three. They agree with the maximum such class to some size, then include none or half (rounded up if necessary) of the permutations in that class at the next size, and then no further permutations.

\begin{proposition}
\label{prop:wedge}
Let $\F \subseteq \S_3$ be the finite class with $\F_3 = \{123, 132, 231, 321\}$. The only infinite uniquely-Wilf class containing $\F$ is $\F^{\uparrow} = \Av(213, 312)$.
\end{proposition}

\begin{proof}

That $\Av(213, 312)$ is uniquely-Wilf is proven in \cite{AlbertLi} so we provide only a short sketch of the argument here for the reader's convenience. Permutations in this class have the wedge-shaped structure shown below:

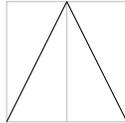
\vspace{\parskip}

\centerline{
\begin{tikzpicture}[scale=0.8]
\draw [color=lightgray, xstep=1.0cm,ystep=2.0cm]
(0,0) grid (2,2);
\draw (0,0) --(1,2);
\draw (1,2) --(2,0);
\end{tikzpicture}
}

The permutations of size $n$ in the class are naturally encoded by binary sequences of length $n-1$ over the alphabet $\{L, R\}$ obtained by reading the permutation from bottom to top and noting which ``side'' of the picture the corresponding element belongs. The containment relation in the class is exactly the subsequence containment relation for these words. Consider two words $\alpha L$ and $\beta R$ where $\alpha$ and $\beta$ are words of the same length. A word contains $\alpha L$ but none of its prefixes do if and only if it is of the form $A R^k L$ where $A$ contains $\alpha$ but none of its prefixes do and  $k \geq 0$. But similarly a word contains $\beta R$ but none of its prefixes do if and only if it is of the form $B L^k R$ where $B$ contains $\beta$ but none of its prefixes do. This sets up a length-preserving bijection between the minimal words containing $\alpha L$ and those containing $\beta R$ whenever we have such a bijection with respect to $\alpha$ and $\beta$. But of course such a bijection exists when $\alpha$ and $\beta$ are both the empty word, so induction shows that this bijection exists in general for any two words of the same length. This bijection then extends to a length-preserving bijection between all words containing $\alpha L$ and all words containing $\beta R$ since these are obtained from words of the first type by concatenation with an arbitrary word. This establishes that $\Av(213, 312)$ is uniquely-Wilf.

Now suppose that $\E$ is a proper subclass  of $\Av(213, 312)$ that contains $\F_3$. The permutations in $\E$ can be partitioned into their maximal monotone intervals which must be alternately increasing and decreasing when considered from bottom to top. There must be a bound on the number of such intervals (since $\E$ is proper). Therefore the structure of a permutation in $\E$ is determined (essentially) by a finite tuple of non-negative integers summing to its size and the enumeration of $\E$ is polynomially bounded and hence ultimately polynomial (\cite{KK}).

The structure of polynomial classes is well-understood and we will make use of a form introduced in \cite{HV} -- we sketch only the necessary details here and refer to that paper for the full results. Define a \emph{peg permutation} to be a permutation written in one line notation where each of its elements is decorated with a superscript from the set $\{+, -, \bullet\}$. For example, $2^{-} 3^{-} 1^{\bullet}$ is a peg permutation. 

Given a pegged permutation $\tilde{\rho}$ the set of permutations $\Grid(\tilde{\rho})$ consists of all those permutation that can be obtained from $\rho$ by inflating each point into a monotone (possibly empty) interval according to the associated symbol: for $+$ an arbitrary increasing interval, for $-$ an arbitrary decreasing interval, and for $\bullet$ either a single point or nothing at all. For instance $4321$, $2431$ and $432651$ all belong to $\Grid(2^{-} 3^{-} 1^{\bullet})$ but $123$ does not.

Departing from the terminology of \cite{HV} we say that a pegged permutation is \emph{properly pegged} if it contains no two element monotone interval of any of the following forms:
\begin{equation}
\label{eq:forbidden-patterns}
1^{+} 2^{+}, \, 1^{\bullet} 2^{+}, \: 1^{+} 2^{\bullet}, \, 2^{-} 1^{-}, \, 2^{\bullet} 1^{-}, \, 2^{-} 1^{\bullet}.
\end{equation}
To be explicit, the first condition above means that there are no two consecutive elements $a$ and $a+1$ in the pegged permutation both of which are decorated with $+$. We define $\Grid^{f}(\tilde{\rho})$ to be the subset of $\Grid(\tilde{\rho})$ consisting of all inflations of $\rho$ in which any symbol decorated with $+$ or $-$ is inflated into a monotone interval of size at least two, and any symbol decorated with $\bullet$ is replaced by a single point (the superscript stands for ``filled''). Restating the conclusion of Theorem 1.4 of \cite{HV} slightly:

\emph{There exists a finite set $\tilde{H}$ of properly pegged permutations such that $\E$ is the disjoint union of the sets $\Grid^{f}(\tilde{\rho})$ for $\tilde{\rho} \in \tilde{H}$.}

We remark that since $\E$ is a class that, if $\tilde{\rho} \in \tilde{H}$ then in fact $\Grid(\tilde{\rho}) \subseteq \E$ since $\Grid(\tilde{\rho})$ is just the downward closure of $\Grid^{f}(\tilde{\rho})$.

Choose $\tilde{\rho} \in \tilde{H}$ in such a way that the total number, $m$ of $+$ and $-$ decoration on $\tilde{\rho}$ is maximal (among all the elements of $\tilde{H}$) and, among all such, the underlying permutation $\rho$ is maximal in the pattern containment order. Equality is allowed in either case. For instance it might be the case that $2^+ 4^+ 1^\bullet 3^-, 2^- 4^\bullet 1^- 3^+ \in \tilde{H}$ . If no other permutation in $\tilde{H}$ contains more than three non-$\bullet$ decorations and no permutation with three non-$\bullet$ decorations contains 2413 then either of these would be a suitable choice for $\tilde{\rho}$. Form a permutation $\pi \in \Grid^{f}(\tilde{\rho})$ by inflating each such non-$\bullet$ symbol to a monotone interval having at least $K$ elements where $K$ is greater than the square of the largest size of any permutation underlying an element of $\tilde{H}$. This choice of $K$ ensures (easily) that the only covers of $\pi$ in $\E$ are obtained by extending the inflated intervals by one point. In particular $\pi$ has exactly $m$ covers in $\E$. Therefore, as $\E$ is uniquely-Wilf, \emph{all} sufficiently large permutations in $\E$ must have exactly $m$ covers.

Now suppose that $\Grid(1^+ 3^+ 2^\bullet) \subseteq \E$. Then any permutation that consists of a monotone increasing sequence followed by one more element belongs to $\E$ as these all occur as elements of $\Grid(1^+ 3^+ 2^\bullet)$. In particular, for any $n$, the monotone increasing permutation of size $n$ has at least $n+1$ covers in $\E$ which contradicts the observation above that all sufficiently large permutations in $\E$ must have exactly $m$ covers. So it follows by contradiction that $\Grid(1^+ 3^+ 2^\bullet) \not \subseteq \E$. This, together with the fact that $\E$ is a class, implies that no $\tilde{\rho} \in \tilde{H}$ contains a 132 pattern with both the 1 and the 3 carrying $+$~decorations -- we'll abbreviate this as ``no $1^+ 3^+ 2$ pattern''. Similarly, there can be no $\tilde{\rho} \in \tilde{H}$ containing a $2 3^- 1^-$ pattern.

This technique of showing that certain restrictions on the labelings of the permutations in $\tilde{H}$ must be satisfied because otherwise some permutation has ``too many'' covers contradicting the assumption that $\E$ is uniquely-Wilf is used throughout the remainder of the argument.

The underlying permutations of $\tilde{H}$ belong to $\Av(312, 213)$ so if there were three or more non-$\bullet$ decorations then two of them would occur on the same side and would therefore have the same sign. Since the intervals described in (\ref{eq:forbidden-patterns}) cannot occur they could not be consecutive, and this would create one of the forbidden situations above. It follows that there can be at most two non-$\bullet$ decorations on each permutation of $\tilde{H}$ and, if there are two, they must have differing signs.

Therefore, every sufficiently large permutation in $\E$ has either exactly one or exactly two covers. The former case is easily dismissed -- since the increasing permutation has only the increasing cover the only permutation in $\tilde{H}$ carrying a $+$ decoration would have to be $1^+$ and similarly $1^-$ would be the only one containing the $-$ decoration (and all the monotone permutations of $\E$ would be represented by these along with $1^\bullet$). The rest of the permutations (if any) would carry only $\bullet$s so define a finite set -- but if any existed there would a non-monotone permutation with no covers. So $\E = \I \cup \D$ but supposedly $\E_3 = \F_3$ which contains non-monotone permutations.

Supposing that each sufficiently large permutation in $\E$ has exactly two covers also implies that  
\[
\Grid(2^+ 3^{-} 1^{\bullet})\not \subseteq  \E
\]
since otherwise the monotone increasing permutation of size $n$ would then have all of:
\[
1 2 3 \cdots (n+1), \, 2 3 \cdots (n+1) 1, \, \mbox{and} \, 1 2 3 \cdots (n-1) (n+1) n
\]
as covers. Similarly, $\Grid(1^\bullet 3^+ 2^-) \not \subseteq \E$.
Also,
\[
\Grid(1^+ 3^\bullet 4^\bullet 2^-) \not \subseteq \E \, \text{and} \, \Grid(2^+ 4^\bullet 3^\bullet 1^-) \not \subseteq \E
\]
as otherwise the monotone increasing (or decreasing in the second case) permutation would have at least three extensions. 

Putting all these observations together shows that the only possible elements of $\tilde{H}$ containing both a $+$ and a $-$ decoration are $1^+ 2^-$ and $2^+ 1^-$. Only one can occur as, if they both did, increasing permutations would have at least three covers in $\E$. 

Suppose without loss of generality that $1^+ 2^- \in \tilde{H}$. We claim that, other than $1^+$, $1^-$, and $1^\bullet 2^-$ all other patterns in $\tilde{H}$ contain only $\bullet$s as decorations. 

To rule out other patterns containing symbols $x^+$ observe first that $\Grid(2^+ 1^\bullet) \not \subseteq \E$ since if it were, there would be an extra extension to the monotone increasing permutation beyond the two already available from $\Grid(1^+ 2^-)$. It follows, that the only possible $+$~decoration on an element of $\tilde{H}$ can occur on the first (and least) element of a pattern. No decorations of $132$ can belong to $\tilde{H}$ since we can't place $\bullet$s on both the 3 and the 2 as that case is already covered by $1^+ 2^-$ and other labelings are forbidden by  (\ref{eq:forbidden-patterns}) or previous observations. By  (\ref{eq:forbidden-patterns}) we cannot begin $1^+ 2^\bullet$. Finally, $\Grid(1^+ 3^{\bullet} 4^{\bullet} 2^{\bullet})$ also gives extra extensions to the monotone increasing permutation. Any larger labeled pattern includes one of these forbidden cases, and so we've ruled out the presence of a $+$ except for $1^+ 2^-$ and $1^+$. 

Now we want to rule out patterns in $\tilde{H}$ other than $1^-$, $1^\bullet 2^-$ and $1^+ 2^-$ that contain symbols $x^-$. We cannot appeal to symmetry because we've broken that by assuming $\Grid(1^+ 2^-)$ is contained in $\E$. In $\Grid(2^{\bullet} 3^{-} 1^{\bullet})$ we gain an extra extension to the monotone decreasing permutation so no pattern in $\tilde{H}$  can contain a $2 3^- 1$ pattern. Similarly no $2 3 1^-$ patterns can occur. Together these rule out all possibilities of having a symbol $x^-$ other than the ones which are explicitly allowed.

Now $231 \in \E$ but is not in $\Grid(1^+)$, $\Grid(1^-)$, $\Grid(1^\bullet 2^-)$, or $\Grid(1^+ 2^-)$ so its representation must be as $\Grid(2^\bullet 3^\bullet 1^\bullet)$. The same applies for its supposed two covers in $\E$, each of their covers, and so on. However, $\tilde{H}$ is finite, so we have our final contradiction.
\end{proof}

The proof of Proposition \ref{prop:wedge} has obviously been the most intricate of the arguments so far. Fundamentally though it is not so different in character to the previous ones (indeed it would have been possible to absorb those into its proof). Working in a polynomial class is a powerful restriction because of the limit it imposes on the number of extensions of any sufficiently large permutation. This then enables the ``surplus extensions'' technique to rule out many configurations. The second author, as part of his current PhD research (to be submitted in 2020), has proven that the only uniquely-Wilf polynomial classes are the ones that already arise in Proposition \ref{prop:wedge}.

\section{Conclusion}

Our goal of characterising the uniquely-Wilf permutation classes has not been fully realised as the possibility still exists that there may be some which we have not discovered. If so, they will contain all of $\S_3$ and at least 22 elements of $\S_4$. 

It is notable that most of the arguments by contradiction in Proposition \ref{prop:wedge} focused on the existence of surplus extensions to a monotone permutation. The monotone permutations are already somewhat special in the order of pattern containment since they are the only permutations that cover only one other permutation. Therefore it does not seem unreasonable to ask:

\begin{question}
What can be determined about classes in which all non-monotone permutations are Wilf-equivalent?
\end{question}

For polynomial classes, the second author has answered this question completely (forthcoming as part of his PhD thesis in 2020) and, other than the examples already discovered above there are only a few new examples for instance the class $\Av(123, 312)$ which was considered in \cite{AlbertLi}. 

The authors would like to acknowledge the helpful comments of two referees, which have improved the presentation of this work significantly.

\nocite{*}

\end{document}